\documentclass[12pt]{amsart}

\usepackage{graphicx}
\usepackage{fncylab,hyperref,environ}

\usepackage{amssymb, amsmath, amsthm, enumerate}
\usepackage{tikz}

\usetikzlibrary{chains}

\theoremstyle{plain}
\newtheorem{thm}{Theorem}
\newtheorem{rmk}{Remark}
\newtheorem{lem}{Lemma}
\newtheorem{cor}{Corollary}

\newcommand{\idd}{\operatorname{Id}}
\newcommand{\trc}{\operatorname{trace}}
\newcommand{\vll}{\operatorname{vol}}
\newcommand{\Aut}{\operatorname{Aut}}

\tikzset{node distance=2em, ch/.style={circle,draw,on chain,inner sep=2pt},chj/.style={ch,join},every path/.style={shorten >=4pt,shorten <=4pt},line width=1pt,baseline=-1ex}

\let\dlabel=\alabel

\newcommand{\dnode}[2][chj]{%
\node[#1,label={below:\dlabel{#2}}] {};
}

\newcommand{\dnodebr}[1]{%
\node[chj,label={below right:\dlabel{#1}}] {};
}

\newcommand{\dydots}{%
\node[chj,draw=none,inner sep=1pt] {\dots};
}

\makeatletter

\newcounter{asyfigcntr}
\@ifundefined{asy}{%
  \NewEnviron{asy}{%
    \par\stepcounter{asyfigcntr}%
    \includegraphics{\jobname-\theasyfigcntr}%
    \par%
  }
}{}

\begin{document}

\title{The 3-ball is a local pessimum for packing}

\author{Yoav Kallus}
\address{Yoav Kallus, Center for Theoretical Science, Princeton University, Princeton, New Jersey 08544, ykallus@princeton.edu}
\thanks{Final submitted manuscript, to be published in \textit{Advances in Mathematics}, \url{http://dx.doi.org/10.1016/j.aim.2014.07.015}}

\date{\today}

\begin{abstract}
It was conjectured by Ulam that the ball has the lowest optimal packing
fraction out of all convex, three-dimensional solids. Here we prove that
any origin-symmetric convex solid of sufficiently small asphericity
can be packed at a higher efficiency than balls. We also show that
in dimensions $4,5,6,7,8,$ and $24$ there are origin-symmetric 
convex bodies of arbitrarily small asphericity that cannot be packed
using a lattice as efficiently as balls can be.
\end{abstract}

\maketitle

\section{Introduction}

When Martin Gardner's \textit{New Mathematical Diversions},
collecting some of his \textit{Mathematical Games} columns, was reprinted in 1995,
the column on ``Packing Spheres'' appeared with a postscript,
in which Gardner writes
``Stanislaw Ulam told me in 1972 that he suspected the sphere was
the worst case of dense packing of identical convex solids, but that
this would be difficult to prove'' \cite{Gardner95}. In other
words, Ulam ``suspected that spheres, in their densest packing,
allow more empty space than the densest packing of any identical convex
solids'' \cite{Gardner01}.
For the purposes of
this article, we restrict our attention only to solids that are origin-symmetric
and only to lattice packings.
Since too many conjectures already bear the name
\textit{Ulam's conjecture}, it would be
appropriate to refer to this one as \textit{Ulam's packing conjecture},
or perhaps even as \textit{Ulam's last conjecture}, seeing how it too was
published posthumously and how the space in Gardner's postscript evidently
could not fit any motivation for Ulam's suspicion. 

The suspicion is not hard to motivate naively
by the fact that the sphere is the most symmetric solid and therefore also
the least free: in placing a sphere in space there are only three
degrees of freedom, compared to five in the case of any other solid of revolution,
and six in the case of any other solid. Therefore, it is natural to suspect that
if we break the rotational symmetry of the sphere,
we introduce more freedom, which could be used to tighten up the packing
and bring the packing fraction above the optimal packing
fraction of spheres.

That in the plane the circle can be packed more densely than some other
domains belies this naive motivation, which one would expect to work
equally well in the plane.
A packing of circles can cover $0.9068\ldots$ of the plane,
while a packing of regular octagons cannot cover more than $0.9061\ldots$.
In 1934, Reinhardt conjectured that a smoothed octagon he constructed has
the lowest optimal packing fraction of any origin-symmetric convex domain
\cite{Reinhardt}. A packing of smoothed octagons
can cover at most $0.9024\ldots$ of the plane.
In fact, the circle is not even a \textit{local} pessimum (the opposite of an optimum). There
are origin-symmetric shapes arbitrarily circular (the outradius
and inradius both being arbitrarily close to $1$) that cannot be packed as efficiently as circles \cite{mahler95}.
Ulam's conjecture implies that this would not be the case in three dimensions. The
question then is in which dimensions, if in any, the ball is a local pessimum.
In this article, we answer this question for all dimensions in which
the highest lattice packing fraction for spheres is known.

In Section 2 we provide some preliminaries about convex bodies and admissible lattices.
In Section 3 we introduce the notions of perfection and eutaxy and their implications
for the question at hand. We observe that these implications are strong enough to determine
that the ball is not a local pessimum in dimensions $4,5,6,7,8,$ and $24$. The
3-dimensional case is treated in Section 4, where the ball is proved to be a local pessimum.

\section{Convex Bodies and Admissible Lattices}

An $n$-dimensional \textit{convex body} is a convex, compact subset of $\mathbb{R}^n$
with a nonempty interior. A planar body is a \textit{domain}, and a 3-dimensional
body is a \textit{solid}. A body $K$ is \textit{symmetric about the origin} (or origin-symmetric) if 
$-K=K$. In this article we discuss only such bodies, and
we will implicitly assume that every body mentioned is symmetric about the origin.
We denote by $B^n$ the Euclidean unit ball of $\mathbb{R}^n$.
The space of origin-symmetric convex bodies $\mathcal K^n_0$ in $\mathbb{R}^n$
is a metric space equipped with the Hausdorff metric
$\delta_H(K,K')=\min\{\varepsilon : K\subseteq K'+\varepsilon B^n, K'\subseteq K+\varepsilon B^n\}$.
The set of bodies $K$ satisfying $a B^n\subseteq K\subseteq b B^n$ for $b>a>0$ is compact \cite{Gruber}.

Let $S^{n-1} = \partial B^n$ be
the unit sphere. The value of the \textit{support function} of an $n$-dimensional body in the
direction $\mathbf{x}\in S^{n-1}$ is given by
$h_K(\mathbf{x}) = \max_{\mathbf{y}\in K} \langle\mathbf{x},\mathbf{y}\rangle$.
The half-space $H_K(\mathbf{x})=\{\mathbf{y}:\langle\mathbf{x},\mathbf{y}\rangle\le h_K(\mathbf{x})\}$
contains $K$, and a body is uniquely determined by its support function since
$K=\bigcap_{\mathbf{x}\in S^{n-1}} H_K(\mathbf{x})$.
Similarly, the value of \textit{radial function} of a body in the direction $\mathbf{x}\in S^{n-1}$
is given by $r_K(\mathbf{x}) = \max_{\lambda\mathbf{x}\in K}\lambda$.
Again, a body is uniquely determined by its radial function.
For origin-symmetric bodies, both the support function and the radial function are
even functions.

An $n$-dimensional \textit{lattice} is the image of the integer lattice $\mathbb{Z}^n$
under some non-singular linear map $T$. The determinant $d(\Lambda)$ of a lattice $\Lambda=T\mathbb{Z}^n$
is the volume of the image of the unit cube under $T$ and is given by $d(\Lambda)=|\det T|$.
The space $\mathcal L^n$ of $n$-dimensional lattices can be equipped with the
metric $\delta(\Lambda,\Lambda')=\min\{||T-T'|| : \Lambda = T\mathbb{Z}^n, \Lambda'= T'\mathbb{Z}^n\}$,
where $||\cdot||$ is the Hilbert-Schmidt norm.
A lattice is called \textit{admissible} for a body $K$ if it intersects the interior of $K$ only at the origin.
A lattice $\Lambda$ is admissible for $K$ if and only if $\{K+2\mathbf{l}:\mathbf{l}\in\Lambda\}$
is a packing, i.e., $K+2\mathbf{l}$ and $K+2\mathbf{l}'$, have disjoint interiors
for $\mathbf{l},\mathbf{l}'\in\Lambda$, $\mathbf{l}\neq\mathbf{l}'$.
The fraction of space covered by this packing is $2^{-n} \vll K / d(\Lambda)$.
The set of lattices $\Lambda$, such that $\Lambda$ is admissible for a body $K$ and $d(\Lambda)<a$ for some
$a>0$, is a compact set \cite{Gruber}.

The \textit{critical determinant} $d_K$ is the minimum, necessarily attained due to compactness, of all
determinants of lattices admissible for $K$.
A lattice attaining this minimum is called a \textit{critical lattice} of $K$.
Clearly, if $K'\subseteq K$, then
$d_{K'}\le d_{K}$. If this inequality is strict whenever $K'$ is a proper subset of $K$,
we say that $K$ is an \textit{irreducible} body. The optimal packing
fraction for $K$ is $\varphi(K) = 2^{-n} \vll K / d_K$. Note that $\varphi(TK)=\varphi(K)$ for any
nonsingular linear transformation $T$. Therefore, we may define $\varphi$ as a function over the
space of linear classes of $n$-dimensional bodies, equipped with the Banach-Mazur distance
$\delta_{BM}([K],[L])=\min\{t:L'\subseteq K'\subseteq e^t L',K'\in[K],L'\in[L]\}$.
Since this space is compact, there must be a body $K$ with the lowest possible
optimal packing fraction among all $n$-dimensional bodies. Such a body is called
a \textit{global pessimum} for packing. If a body attains the lowest possible optimal
packing fraction among bodies in a neighborhood of itself with respect to Hausdorff distance,
then we say it is a \textit{local pessimum} for packing.
A locally pessimal body is necessarily irreducible, but the converse
is not necessarily true.

In two dimensions, Reinhardt's smoothed octagon is known to be locally pessimal and is conjectured to be
globally pessimal \cite{Nazarov}. The unit disk in two dimensions is irreducible
but not locally pessimal \cite{mahler95}. Below we show that the unit ball is
locally pessimal for $n=3$, irreducible but not locally pessimal for $n=4$ and $5$,
and reducible for $n=6,7,8,$ and $24$.

\section{Perfection and Eutaxy}

Let $S\subset S^{n-1}$ be a finite, origin-symmetric set of unit vectors (that is
$\mathbf{x}\in S$ if and only if $-\mathbf{x}\in S$). Every vector of $S$ defines a
projection map $P_\mathbf{x}(\cdot)=\langle\cdot,\mathbf{x}\rangle\mathbf{x}$. If these
projection maps span the space $\mathrm{Sym}^n$ of symmetric linear maps $\mathbb R^n\to\mathbb R^n$,
we say that $S$ is a 
\textit{perfect} configuration. The space $\mathrm{Sym}^n$,
equipped with inner product $\langle Q,Q'\rangle = \trc (QQ')$ is isomorphic to $\mathbb R^{n(n+1)/2}$.
Therefore, a perfect configuration must have at least $n(n+1)$ vectors \cite{perfect}.

If there are numbers $\upsilon_\mathbf{x}$, $\mathbf{x}\in S$,
such that 
\begin{equation}
    \upsilon_\mathbf{x}=\upsilon_{-\mathbf{x}} \text{ and } \idd = \sum_{\mathbf{x}\in S} \upsilon_\mathbf{x} P_\mathbf{x}\text,
    \label{eutax}
\end{equation}
then we say that $S$ is a \textit{weakly eutactic} configuration and we call the coefficients
$\upsilon_\mathbf{x}$ eutaxy coefficients \cite{perfect}. A configuration is \textit{eutactic} (resp.\ \textit{semi-eutactic})
if it has positive (resp.\ non-negative) eutaxy coefficients.
For reasons that would become apparent with Corollary \ref{vorcor}, a configuration that is perfect
and eutactic is called \textit{extreme}.
A configuration $S$ is semi-eutactic if and only if the identity map is in the closed cone generated by
$\lbrace P_\mathbf{x}:\mathbf{x}\in S\rbrace$, and extreme if and only if the identity is
in the interior of this cone.

Suppose that $S$ is extreme, then
we say that $S$ is redundantly semi-eutactic (resp.\ redundantly extreme) if the configuration
$S\setminus\{\mathbf{x},-\mathbf{x}\}$ is semi-eutactic (resp.\ extreme)
for all $\mathbf{x}\in S$. On the other hand, if $S$ is extreme and has the minimal number
of vectors $n(n+1)$, we say the configuration is \textit{minimally extreme}. The eutaxy
coefficients of a minimally extreme configuration are unique and the configuration
is not redundantly semi-eutactic.

The \textit{minimal norm} $m(\Lambda)$ of a lattice $\Lambda$ is the minimum of
$\langle \mathbf{x},\mathbf{x}\rangle$ over
$\mathbf{x}\in\Lambda\setminus\{0\}$, and the set of points $S(\Lambda)\subseteq\Lambda$ that achieve this value
are the \textit{minimal vectors} of $\Lambda$. We say of a lattice $\Lambda$
that it is perfect (resp.\ eutactic, semi-eutactic, and so on), if 
the configuration of its minimal vectors, scaled to lie on the unit sphere, is perfect (resp.\ eutactic, semi-eutactic, and so on).

In dimensions $n\le 8$ and $n=24$, the critical lattices of $B^n$ are known.
Admissibility for $B^n$ is invariant under rotations, and so the critical
lattice can be unique only up to rotations, which in fact it is in these dimensions. For $n=2,3,4,5,6,7,8$,
the critical lattices of $B^n$ are given by the root lattices
$A_2$, $D_3$, $D_4$, $D_5$, $E_6$, $E_7$, and $E_8$ \cite{SPLAG}.
The root lattices are most easily specified (up to rotation) in terms of their Dynkin diagrams:
\begin{equation}
\begin{aligned}
A_n &&& 
\begin{tikzpicture}[start chain]
\dnode{1} \dnode{2} \dydots \dnode{n-1} \dnode{n}
\end{tikzpicture}&&\text{for }n=2,3,\ldots\text,\\
D_n &&&
\begin{tikzpicture}
\begin{scope}[start chain]
\dnode{1} \dnode{2} \node[chj,draw=none] {\dots};
\dnode{n-2} \dnode{n-1}
\end{scope}
\begin{scope}[start chain=br going above]
\chainin(chain-4);
\dnodebr{n}
\end{scope}
\end{tikzpicture}&&\text{for }n=3,4,\ldots\text,\\
E_n &&&
\begin{tikzpicture}
\begin{scope}[start chain]
\dnode{1} \node[chj,draw=none] {\dots};
\dnode{n-3} \dnode{n-2} \dnode{n-1}
\end{scope}
\begin{scope}[start chain=br going above]
\chainin(chain-3);
\dnodebr{n}
\end{scope}
\end{tikzpicture}&&\text{for }n=6,7,8\text.
\end{aligned}
\end{equation}
Here the nodes correspond to generating vectors of the lattice, i.e.,
$\Lambda = \lbrace\sum_{i=1}^n z_i \mathbf{a}_i: z_i\in\mathbb{Z}\rbrace$,
and their inner products are determined by the edges: $\langle\mathbf{a}_i,\mathbf{a}_j\rangle=1,-\tfrac12,0$
respectively when $i=j$, $\mathbf{a}_i$ is connected by an edge to $\mathbf{a}_j$, and otherwise. 
Note that $A_3$ and $D_3$ are identical. The lattice $A_n$ can be embedded isometrically as
an $n$-dimensional sublattice of the $(n+1)$-dimensional integer lattice $\mathbb{Z}^{n+1}$,
namely
\begin{equation}
    A_n \simeq \lbrace 2^{-1/2}\mathbf{z} : \mathbf{z}\in\mathbb{Z}^{n+1}, \sum_{i=1}^{n+1} z_i = 0 \rbrace\text.
    \label{an}
\end{equation}
A particular rotation of the lattice $D_n$ is given by
\begin{equation}
    D_n = \lbrace 2^{-1/2}\mathbf{z} : \mathbf{z}\in\mathbb{Z}^{n}, \sum_{i=1}^{n} z_i \in 2\mathbb{Z} \rbrace\text.
    \label{dn}
\end{equation}

In $24$ dimensions, the Leech lattice $\Lambda_{24}$ is the unique critical lattice of $B^n$ \cite{leechoptimal}.
The Leech lattice can be constructed as a union of two translates, $\Lambda_{24}=L\cup(L+\mathbf{x})$,
where
\begin{equation}\begin{aligned}
	L &= \lbrace 2^{-3/2}\mathbf{z}: \mathbf{z}\in\mathbb{Z}^{24}, \sum_{i=0}^{24}z_i\in 4\mathbb{Z}, \mathbf{z}-\mathbf{c} \in (2\mathbb{Z})^n, \mathbf{c}\in\mathcal{C} \rbrace\\
	\mathbf{x} &= 2^{-5/2}(-3,1,1,\ldots,1)\text,
	\label{leech}
\end{aligned}\end{equation}
and $\mathcal C\subseteq\lbrace0,1\rbrace^{24}$ is the extended binary Golay code of size $2^{12}$ and minimum
weight $8$ \cite{SPLAG}. All the root lattices and the Leech lattice are extreme \cite{SPLAG}.
In fact, since the automorphism
group $\Aut(\Lambda)$ acts transitively on the minimal vectors $S(\Lambda)$ in all these cases,
it is possible to pick all the eutaxy coefficients equal to $n/|S(\Lambda)|$ \cite{SPLAG}.
We now establish which of these lattices are minimally extreme, redundantly extreme, and
redundantly semi-eutactic.

\begin{lem}\label{rootlem}
    The lattices $A_n$ are minimally extreme for all $n=2,3,\ldots$. The lattices
    $D_n$ are redundantly semi-eutactic and not redundantly extreme for all $n=4,5,\ldots$.
    The lattices $E_6$, $E_7$, $E_8$, and $\Lambda_{24}$ are redundantly extreme.
\end{lem}

\begin{proof}
    The minimal vectors of the embedding of $A_n$ in $\mathbb{R}^{n+1}$ given in \eqref{an} are
    all the vectors derived from 
    $2^{-1/2}(1,-1,0,0,\ldots,0)$ by coordinate permutations. There are exactly
    $n(n+1)$ such vectors and therefore $A_n$ is minimally extreme.
    
    The minimal vectors of $D_n$, as embedded in $\mathbb{R}^n$ by \eqref{dn}, are all $2n(n-1)$ permutations
    of vectors of the form $2^{-1/2}(\pm1,\pm1,0,0,\ldots,0)$. We want to show that for all
    $\mathbf{x}\in S(D_n)$, the configuration $S(D_n)\setminus\lbrace\pm\mathbf{x}\rbrace$
    is semi-eutactic but not extreme. Note that because $\Aut(D_n)$ acts transitively
    on $S(D_n)$, $S(D_n)\setminus\lbrace\pm\mathbf{x}\rbrace$ is semi-eutactic (extreme)
    if and only if $S(D_n)\setminus\lbrace\pm\mathbf{x}'\rbrace$ is. Consider
    the linear system \eqref{eutax} for the eutaxy coefficients of $S(D_n)$.
    When $n\ge4$, it has more variables than equations
    and has at least one positive solution, $\mathbf{u}_0 = (\tfrac1{2(n-1)})_{\mathbf{x}\in S(D_n)}$.
    Therefore, the solution space must contain a line $\mathbf{u}_0+t \mathbf{u}_1$, which in turn must
    contain a non-negative solution with at least one pair of vanishing coefficients $\upsilon_{\pm\mathbf{x}_0}=0$.
    Therefore, $S(D_n)\setminus\lbrace\pm\mathbf{x}_0\rbrace$ is semi-eutactic, and
    by symmetry $S(D_n)$ is redundantly semi-eutactic. Now let
    $Q = \sum_{\mathbf{x}\in S(D_n)}\upsilon_{\mathbf{x}}P_\mathbf{x}$.
    If $\upsilon_{\pm2^{-1/2}(1,1,0,0,\ldots,0)} = 0$, then $Q_{12} = -\upsilon_{2^{-1/2}(1,-1,0,0,\ldots,0)}$.
    Therefore, there is no set of all-positive eutaxy coefficients for the configuration
    $S(D_n)\setminus\lbrace\pm2^{-1/2}(1,1,0,0,\ldots,0)\rbrace$ and $S(D_n)$ is not
    redundantly extreme.

    For the lattices $\Lambda=E_7$, $E_8$, and $\Lambda_{24}$, it is enough to note
    that they contain as sublattices of equal minimal norm the root lattices
    $\Lambda'=A_7$, $D_8$, and $D_{24}$, respectively, which are themselves extreme. Since $\idd$ is
    already in the interior of the cone $\lbrace \sum \upsilon_\mathbf{x} P_\mathbf{x} : \mathbf{x}\in S(\Lambda')\rbrace$,
    it is also in the interior of
    $\lbrace \sum \upsilon_\mathbf{x} P_\mathbf{x} : \mathbf{x}\in S(\Lambda)\setminus\lbrace\pm\mathbf{x}_0\rbrace\rbrace$,
    when $\mathbf{x}_0\in S(\Lambda)\setminus S(\Lambda')$, but by symmetry, the same is true
    for all $\mathbf{x}_0\in S(\Lambda)$.
    For the lattice $E_6$, a direct calculation for a particular choice $\mathbf{x}_0\in S(E_6)$,
    which we omit, shows that
    $\idd = \sum_{\mathbf{x}\in S} \upsilon_\mathbf{x}P_\mathbf{x}$, where $\upsilon_\mathbf{x} = \tfrac1{15},\tfrac1{10},0$
    respectively when $\langle\mathbf{x},\mathbf{x}_0\rangle = 0, \pm\tfrac12, \pm1$,
    so $S(E_6)\setminus\lbrace\pm\mathbf{x}_0\rbrace$ is eutactic.
    Another direct calculation shows that $\mathrm{Sym}^6$ is spanned by $P_\mathbf{x}$,
    $\mathbf{x}\in S(E_6)\setminus\lbrace\pm\mathbf{x}_0\rbrace$.
    Therefore, the lattices $E_6$, $E_7$, $E_8$, and $\Lambda_{24}$ are redundantly extreme.
\end{proof}

We establish a relation between eutaxy properties and the existence or non-existence
of linear maps that satisfy certain conditions. For this purpose we recall
the fundamental theorem of linear algebra, which in our case implies that
there exists a symmetric linear map $Q$ such that
\begin{equation*}\begin{aligned}
\langle\mathbf{x},Q\mathbf{x}\rangle = \langle P_\mathbf{x},Q\rangle &= a_\mathbf{x} \\
\trc Q = \langle\idd,Q\rangle &= -a_0\text,
\end{aligned}\end{equation*}
if and only if $\sum_{\mathbf{x}\in S} a_\mathbf{x}b_\mathbf{x}+a_0 b_0=0$ whenever
the coefficients $b_0$, $(b_\mathbf{x})_{\mathbf{x}\in S}$ are a solution to the equation
\begin{equation}\label{ALP1}
\sum_{\mathbf{x}\in S} b_\mathbf{x} P_\mathbf{x} - b_0 \idd = 0\text,
\end{equation}
In other words, the
space of possible $a$'s is the orthogonal complement of the space of
possible $b$'s. Also recall that a subspace of $\mathbb{R}^{m}$ contains a
positive vector, if and only if its orthogonal complement does not contain
a non-zero, non-negative vector. With these facts in mind we prove the following
useful theorem:

\begin{thm}\label{linmaps}
\begin{enumerate}[(i)]

\item\label{pe1} If $S$ is extreme, there exists $\varepsilon>0$ and $C>0$ such that
whenever $T$ is a linear map, $||T-\idd||<\varepsilon$,
and $||T\mathbf{x}||\ge||\mathbf{x}||$ for all $\mathbf{x}\in S$, then $\det T\ge1+ C ||T^TT-\idd||$.

\item\label{pe2} If $S$ is not extreme then for arbitrarily small
$\varepsilon>0$ there exists a linear map $T$ satisfying $||T-\idd||<\varepsilon$,
$||T\mathbf{x}||\ge||\mathbf{x}||$ for all $\mathbf{x}\in S$, and $\det T<1$.

\item\label{se} $S$ is semi-eutactic if and only if there exists $\varepsilon>0$ and $C>0$ such that
whenever $T$ is a linear map, $||T-\idd||<\varepsilon$,
and $||T\mathbf{x}||\ge||\mathbf{x}||$ for all $\mathbf{x}\in S$,
then $\det T>1-C||T-\idd||^2$.

\item\label{ce} Let $S$ be minimally extreme with eutaxy coefficients $\upsilon_\mathbf{x}$
and let numbers $\omega_\mathbf{x}$ be given for $\mathbf{x}\in S$ such that
$\omega_\mathbf{x}=\omega_{-\mathbf{x}}$,
then there exists a symmetric linear map $Q\in\mathrm{Sym}^n$ satisfying
$\langle\mathbf{x},Q\mathbf{x}\rangle=\omega_\mathbf{x}$
and $\trc Q = \omega_0$ if and only if $\omega_0=\sum_{\mathbf{x}\in S} \omega_\mathbf{x} \upsilon_\mathbf{x}$.

\end{enumerate}
\end{thm}

\begin{proof}
\begin{enumerate}[(i)]
\item
By eutaxy, there exist positive coefficients $b_\mathbf x$ and $b_0$ satisfying
\eqref{ALP1}. Therefore, every symmetric map $Q$ such that
\begin{equation}\label{LP1}
\langle P_\mathbf{x}, Q\rangle\ge0\text{ , for all }\mathbf x\in S
\end{equation}
and $\langle\idd, Q\rangle\le0$ must satisfy $\langle P_\mathbf{x}, Q\rangle=0$ for all $\mathbf{x}\in S$,
and therefore, by perfection $Q=0$. Thus, if $Q\neq 0$  satisfies \eqref{LP1} then $\langle Q,\idd\rangle>0$.
In fact, by compactness of the unit sphere in $\mathrm{Sym}^n$, there must be a positive number $C'>0$ such
that $\langle Q,\idd\rangle \ge C' ||Q||$ whenever $Q$ satisfies \eqref{LP1}.
Now let $T$ be a linear map such that $||T-\idd||<\varepsilon$
and $||T\mathbf{x}||\ge||\mathbf{x}||$ for all $\mathbf{x}\in S$.
Then $Q=T^T T-\idd$ satisfies \eqref{LP1}, and therefore $\trc Q = \langle Q,\idd\rangle \ge C' ||Q||$.
For each $C'$, there exists $\varepsilon$ and $C$ such that $\det T = \sqrt{\det(\idd+Q)} \ge 1+ C ||Q||$
whenever $||T-\idd||<\varepsilon$.

\item
If $S$ is not eutactic, there is no solution with positive coefficients to $\eqref{ALP1}$, and therefore there is
a non-zero solution to $\eqref{LP1}$ with $\langle Q,\idd\rangle\le0$. Similarly, if $S$ is eutactic but not perfect,
any map $Q$ in the complement of the linear span of $P_\mathbf{x}$, $\mathbf{x}\in S$, satisfies
$\eqref{LP1}$ and $\langle Q,\idd\rangle=0$. Consider the map $T_\alpha=\sqrt{\idd+\alpha Q}$, where the square
root indicates the unique positive-definite square root. Note that $||T_\alpha\mathbf{x}||\ge||\mathbf{x}||$,
$\det T_\alpha = \sqrt{\det(1+\alpha Q)}<(1+\tfrac{\alpha}{n}\trc Q)^{n/2}\le 1$,
and $||T_\alpha-\idd||$ can be arbitrarily small.

\item
Suppose $S$ is semi-eutactic, then there exists a non-zero solution to \eqref{ALP1} with non-negative coefficients.
Therefore, there is no symmetric map $Q$ such that
\begin{equation}\label{LP2}
\langle Q,P_\mathbf{x}\rangle>0\text{ , for all }\mathbf x\in S
\end{equation}
and $\langle Q,\idd\rangle<0$. In fact there must also be no symmetric map $Q$ satisfying \eqref{LP1} and $\langle Q,\idd\rangle<0$,
since $Q'=Q+\tfrac{1}{2n}\langle Q,\idd\rangle\idd$ would satisfy \eqref{LP2} and $\langle Q',\idd\rangle<0$.
Again let $T$ be a linear map such that $||T-\idd||<\varepsilon$ and $||T\mathbf{x}||\ge||\mathbf{x}||$ for all $\mathbf{x}\in S$.
Then $Q=T^T T-\idd$ satisfies \eqref{LP1}, and therefore $\trc Q\ge0$. If $\varepsilon$
is sufficiently small, then $\det T = \sqrt{\det(\idd+Q)} > 1-C||T-\idd||^2$.

If $S$ is not semi-eutactic, then there is no non-zero solution to \eqref{ALP1} with non-negative coefficients
and therefore, there exists a symmetric map $Q$ satisfying \eqref{LP2} and $\trc Q<0$.
Again, consider the map $T_\alpha=\sqrt{\idd+\alpha Q}$ and note that $||T_\alpha\mathbf{x}||>||\mathbf{x}||$,
$\det T_\alpha = \sqrt{\det(1+\alpha Q)}\le(1+\tfrac{\alpha}{n}\trc Q)^{n/2}\le 1+\alpha\trc Q$,
and $||T_\alpha-\idd||$ can be arbitrarily small.

\item
Since $S$ is minimally extreme, there is a unique solution up to scaling to \eqref{ALP1}.
Therefore, there is a symmetric map $Q$ satisfying
$\langle\mathbf{x},Q\mathbf{x}\rangle=\langle Q,P_\mathbf{x}\rangle\ge\omega_\mathbf{x}$
and $\trc Q=\langle Q,\idd\rangle\le\omega_0$ if and only if
$\sum_{\mathbf{x}\in S}\upsilon_\mathbf{x} \omega_\mathbf{x}-\omega_0=0$.
In fact, the map $Q$ is unique and depends linearly on the variables $\omega_\mathbf{x}$.
\end{enumerate}
\end{proof}

\begin{cor}\label{vorcor} (Voronoi) $\Lambda$ locally minimizes $d(\Lambda)$ among
admissible lattices for $B^n$ if and only if
$m(\Lambda)=1$ and $\Lambda$ is extreme.\end{cor}

\begin{proof} $\Lambda$ is a local minimum among admissible lattice for
$B^n$ if and only if there exists $\varepsilon>0$ such that
whenever $||T-\idd||<\varepsilon$ and $||T\mathbf{x}||\ge1$ for
all $\mathbf{x}\in S(\Lambda)$ then $d(T\Lambda)\ge d(\Lambda)$ \cite{perfect}.\end{proof}

\begin{cor} If $\Lambda$ is extreme and $m(\Lambda)=1$ then there exists $\varepsilon>0$ and $C>0$, such
that whenever $\Lambda'$ is admissible for $B^n$ and $\delta(\Lambda',\Lambda)<\varepsilon$ then
$d(\Lambda')-d(\Lambda)\ge C\min_{U}\delta(\Lambda',U\Lambda)$, where the minimum is over
all orthogonal maps $U$.\end{cor}

\begin{proof}If $\Lambda' = T\Lambda$ then $\min_{U}\delta(\Lambda',U\Lambda) \le \delta(\sqrt{T^TT}\Lambda,\Lambda)
\le ||\sqrt{T^TT}-\idd|| \le C' ||T^T T -\idd||$.
\end{proof} 

A consequence of the last corollary, since there can be only finitely
many (up to rotation) local minima of the determinant, 
is that there exists $\varepsilon>0$ and $C>0$ such that if $\Lambda$ is admissible for $B^n$
and $d(\Lambda)-d_{B^n}<\varepsilon$ then there exists a lattice $\Lambda'$,
critical for $B^n$, such that $\delta(\Lambda,\Lambda')<C(d(\Lambda)-d_{B^n})$.
Also, if $\Lambda$ is critical
for a nearly spherical body $K$ satisfying $(1-\varepsilon')B^n\subseteq K\subseteq (1+\varepsilon')B^n$,
where $\varepsilon'<\varepsilon$ then again there exists a lattice $\Lambda'$,
critical for $B^n$ such that $\delta(\Lambda,\Lambda')<C \varepsilon'$.

\begin{thm} Suppose that $\Lambda$ is the unique critical lattice of $B^n$ up to
rotations. Then $B^n$ is reducible if and only if $\Lambda$ is redundantly extreme.\end{thm}

\begin{proof} Consider the $\varepsilon$-symmetrically truncated ball
$B_\varepsilon=\{\mathbf{x}\in B^n:-1+\varepsilon\le\langle\mathbf{x},\mathbf{p}\rangle\le1-\varepsilon\}$,
where $\mathbf{p}\in S^{n-1}$ is some arbitrarily chosen pole. First assume that
$\Lambda$ is not redundantly extreme. That is, we assume that
there exists $\mathbf{x}_0\in S(\Lambda)$ such that $S(\Lambda)\setminus \{\pm\mathbf{x}_0\}$
is not extreme. We are free to assume that $\Lambda$ is rotated
so that $\mathbf{x}_0=\mathbf{p}$. Then by Theorem \ref{linmaps} (ii), there exists
a linear map $T$ satisfying
$||T\mathbf{x}||\ge1$ for all $\mathbf{x}\in S(\Lambda)\setminus \{\pm\mathbf{p}\}$,
$\det T<1$, and $||T-\idd||$ is arbitrarily small. In fact, if $||T-\idd||$ is small enough,
then $T\mathbf{p}\not\in B_\varepsilon$ and $T\Lambda$ is admissible for $B_\varepsilon$.
Since $d(T\Lambda)<d(\Lambda)$, and since for each proper $K\subset B^n$, there exists $\varepsilon>0$ such that
$K\subset B_\varepsilon\subset B^n$, it follows that $B^n$ is irreducible.

Now suppose that $\Lambda$ is redundantly extreme. Let $T\Lambda$
be a critical lattice of $B_\varepsilon$. $||T-\idd||$ can be
made arbitrarily small by choosing $\varepsilon$ sufficiently small
and appropriately rotating $\Lambda$. If $\varepsilon$
is small enough, then because $T\Lambda$ is admissible for $B_\varepsilon$,
of the vectors of $T S(\Lambda)$ only one pair
$\pm T \mathbf{x}_0$ can be within the interior of the unit ball $B^n$.
Since $\Lambda$ is redundantly extreme, the requirement that $||T\mathbf{x}||\ge1$
whenever $\mathbf{x}\in S(\Lambda)\setminus\{\pm \mathbf{x}_0\}$, necessarily
implies, when $||T-\idd||$ is small enough, that $\det T\ge1$.
Of course, since $d_{B_\varepsilon}\le d_{B^n}$, we have that $\det T=1$ and $B^n$ is reducible.
\end{proof}

\begin{cor}\label{cor2} For $n=6,7,8,$ or $24$, the unit ball $B^n$ is reducible.\end{cor}
\begin{proof}Recall that $E_6$, $E_7$, $E_8$, and $\Lambda_{24}$ are the unique
    critical lattices for $B^n$ for these values of $n$. Recall also from Lemma \ref{rootlem}
    that they are redundantly extreme.\end{proof}

\begin{thm} If $\Lambda$ is the unique critical lattice of $B^n$ up to 
rotations and $\Lambda$ is redundantly semi-eutactic then $B^n$ is not locally pessimal for packing.\end{thm}

\begin{proof}
Let us call the convex hull of $B_{1/100}$ and $(1-\varepsilon)B^n$, the $\varepsilon$-shaved ball $B'_\varepsilon$.
Let $T\Lambda$ be a critical lattice of $B'_\varepsilon$. Again, if $\varepsilon$
is small enough, then of the vectors of $T S(\Lambda)$ only one pair
$\pm T \mathbf{x}_0$ can be within the interior of the unit ball $B^n$.
Since $\Lambda$ is redundantly semi-eutactic, the requirement that $||T\mathbf{x}||\ge1$
whenever $\mathbf{x}\in S(\Lambda)\setminus\{\pm \mathbf{x}_0\}$, necessarily
implies, when $||T-\idd||$ is small enough, that $\det T>1-C||T-\idd||^2$.
However, note that $||T-\idd||<C'\varepsilon$ and $\vll B'_\varepsilon/\vll B^n<1-c\varepsilon$.
Therefore, for all small enough $\varepsilon$, $\varphi(B'_\varepsilon)<\varphi(B^n)$ and the
ball is not a local pessimum.
\end{proof}

\begin{cor}\label{cor3} For $n=4$ or $5$, the unit ball $B^n$ is irreducible but not locally pessimal.\end{cor}
\begin{proof}Recall that $D_4$ and $D_5$ are the unique
    critical lattices for $B^n$ for these values of $n$. Recall also from Lemma \ref{rootlem}
    that they are redundantly semi-eutactic but not redundantly extreme. Therefore, as a corollary
    of the two preceding theorems, $B^n$ in these dimensions is irreducible but not locally pessimal.\end{proof}

Of the dimensions where the optimal packing fraction of the ball is known, the theorems
given and the known results for $n=2$ leave only the case $n=3$ unresolved, where the
unique critical lattice $D_3$ is minimally extreme.
We show next that the critical determinant of a nearly spherical body in
a dimension where the unique critical lattice is minimally extreme (as
is also the case for $n=2$, but presumably not in any other dimension)
is bounded from above in first order by the average
of its support function (or its radial function)
evaluated at the minimal vectors of the critical lattice of the ball,
weighted by the eutaxy coefficients.
The form of the error term depends on whether the support function or
the radial function is used.

\begin{thm}\label{cns1}
Suppose that $\Lambda_0$ is the unique critical lattice of $B^n$, and $\Lambda_0$ is
minimally extreme with eutaxy coefficients $\upsilon_\mathbf{x}$ for $\mathbf{x}\in S(\Lambda_0)$.
\begin{enumerate}[(i)]
\item
Let $h_\mathbf{x}=1+\eta_\mathbf{x} = h_K(\mathbf{x})$ be the values of the support
function of a nearly spherical body $(1-\varepsilon)B\subseteq K\subseteq(1+\varepsilon)B$ 
evaluated at $\mathbf{x}\in S(\Lambda_0)$.
Then for small enough $\varepsilon$ there is a lattice $\Lambda$, admissible for $K$, whose determinant $d(\Lambda)$ satisfies
\begin{equation}\label{det1}\frac{d(\Lambda)}{d(\Lambda_0)} \le  
1 + \sum_{\mathbf{x}\in S(\Lambda_0)}\upsilon_\mathbf{x} \eta_\mathbf{x}
 + C \max_{\mathbf{x}\in S(\Lambda_0)}|\eta_\mathbf{x}|^2\text.\end{equation}

\item
Let $r_\mathbf{x}=1+\rho_\mathbf{x} = r_K(\mathbf{x})$ be the values of the radial
function of a nearly spherical body $(1-\varepsilon)B\subseteq K\subseteq(1+\varepsilon)B$
evaluated at $\mathbf{x}\in S(\Lambda_0)$. 
Then for small enough $\varepsilon$ there is a lattice $\Lambda'$, admissible for $K$, whose determinant $d(\Lambda')$ satisfies
\begin{equation*}
\frac{d(\Lambda')}{d(\Lambda_0)} \le
1 + \sum_{\mathbf{x}\in S(\Lambda_0)}\upsilon_\mathbf{x}\rho_\mathbf{x}
 + \varepsilon' \sum_{\mathbf{x}\in S(\Lambda_0)}|\rho_\mathbf{x}|\text,
\end{equation*}
where $\varepsilon'$ depends on $\varepsilon$ and becomes arbitrarily small
as $\varepsilon\to0$.
\end{enumerate}

\end{thm}

\begin{proof}
\begin{enumerate}[(i)]\item
By Theorem \ref{linmaps} (iv), there exists a symmetric linear map $Q$ such that
$\langle\mathbf{x},Q\mathbf{x}\rangle=\eta_\mathbf{x}$ and 
$\trc Q=\sum_{\mathbf{x}\in S(\Lambda_0)}\upsilon_\mathbf{x}\eta_\mathbf{x}$
and this map depends linearly on the variables $\eta_\mathbf{x}$.
Consider then the lattice $(\idd+Q)\Lambda$. For all $\mathbf{x}\in S(\Lambda_0)$,
the vector $(\idd+Q)\mathbf{x}$ satisfies
$\langle(\idd+Q)\mathbf{x},\mathbf{x}\rangle\ge h_K(\mathbf{x})$ and therefore
lies outside the interior of $K$. Therefore $\Lambda$ is admissible for $K$.
Because $Q$ depends linearly on $\eta_\mathbf{x}$, there is a constant $C$ such
that \eqref{det1} holds.

\item
The situation in this case is similar to the situation in the previous case, except
we must use the values of the radial function in the directions $\mathbf{x}\in S(\Lambda_0)$, instead of
the support function values, to construct an admissible lattice.
By the construction given in the previous case, we can construct a lattice $\Lambda=T\Lambda_0$,
in general not admissible, such that $\langle\mathbf{x},T\mathbf{x}\rangle=1+\rho_\mathbf{x}$
and
\begin{equation*}\frac{d(\Lambda)}{d(\Lambda_0)} \le  
1 + \sum_{\mathbf{x}\in S(\Lambda_0)}\upsilon_\mathbf{x}\rho_\mathbf{x}
 + C \max_{\mathbf{x}\in S(\Lambda_0)}|\rho_\mathbf{x}|^2\text.\end{equation*}

\begin{figure}
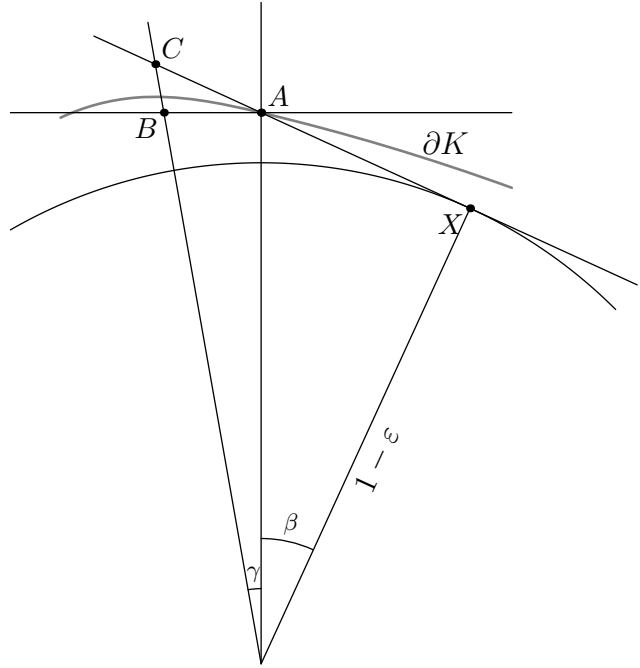
\begin{center}
\begin{asy}
    import cseblack;
    import olympiad;
    usepackage("amssymb");

    size(250);

    pair O=origin, A=D("A",(0,1.1),NE);
    real gamma=10;
    MC("\partial K",D(A+(-0.4,-0.01)..A+(-0.25,0.03)..A..A+(0.5,-0.15),grey+linewidth(1)),0.85,N);
    D(A);

    D(CR(O,1,45,120));
    pair X=D("X",tangent(A,O,1,2),SW),Y=dir(90+gamma);
    path T=D(L(A,X,0.8));
    MC(rotate(angle(X)*180/pi)*"1-\varepsilon",D(O--X),SE);
    MA("\beta",X,O,A,0.25);
    MA("\gamma",A,O,Y,0.15);
    D(L(O,A,0,0.2));
    path Q=D(L(O,Y,0,0.3));
    path R=D((-0.5,A.y)--(0.5,A.y));
    pair B=D("B",IP(Q,R),SW), C=D("C",IP(Q,T),NE);
\end{asy}
\caption{\label{tanfig}
Illustration of the construction given in the proof
of Theorem \ref{cns1} to bound the dilation factor needed
to ensure that the original point $B$ when dilated to $C$
lies outside the body $K$.}
\end{center}\end{figure}

We wish to dilate this lattice by a factor $1+\alpha$ so that $\Lambda'=(1+\alpha)\Lambda$ is admissible. 
Therefore for all $\mathbf{x}\in S(\Lambda_0)$ we must have
$$\alpha\ge\alpha_\mathbf{x}=\frac{r_K(T\mathbf{x}/||T\mathbf{x}||)-||T\mathbf{x}||}{||T\mathbf{x}||}\text.$$
As we only want to use information about the radial function evaluated at $S(\Lambda_0)$
(and the fact that the radial function is bounded between $1-\varepsilon$ and $1+\varepsilon$),
we bound the dilation factor as illustrated in Figure \ref{tanfig}.
In the plane containing the
origin $O$, $(1+\rho_\mathbf{x})\mathbf{x}$ (denoted $A$ in the figure),
and $T\mathbf{x}$ (denoted $B$), we draw the tangent $AX$ from
$A$ to the circle of radius $1-\varepsilon$ about
the origin in the direction away from $B$. Note that $B$ lies on the line through
$A$ perpendicular to $OA$.
Since $\rho_\mathbf{x}<\varepsilon$, the angle $\beta=\widehat{AOX}$
satisfies $\beta\le\arccos\frac{1-\varepsilon}{1+\varepsilon}\le2\sqrt{\varepsilon}$.
By convexity, the continuation of the tangent from $A$ away from
the circle must lie outside of $K$. We mark the intersection of the
tangent and the ray $OB$ as $C$. Then either $\alpha_\mathbf{x}\le0$,
or the boundary of $K$ intersects the ray $OB$ between $C$
and $B$. Since $T$ depends linearly on $\rho_\mathbf{x}$, the angle $\gamma=\widehat{AOB}$ satisfies
$\gamma\le C\sum_{\mathbf{x}\in S(\Lambda_0)}|\rho_\mathbf{x}|$ for some constant $C$.
By the law of sines, we have
\begin{equation*}
|BC| = \frac{|AB|\sin(\beta)}{\cos(\gamma+\beta)}
\le \frac{(1+\varepsilon)\gamma\beta}{1-\frac{1}{2}(\beta+\gamma)^2} \le (1-\varepsilon)\varepsilon' \sum_{\mathbf{x}\in S(\Lambda_0)}|\rho_\mathbf{x}|\text,
\end{equation*}
where $\varepsilon'$ depends on $\varepsilon$ and becomes arbitrarily small
as $\varepsilon\to0$. Therefore, if we let $\alpha = \varepsilon'\sum_{\mathbf{x}\in S(\Lambda_0)}|\rho_\mathbf{x}|$,
we have that $\alpha\ge\alpha_\mathbf{x}$
and $(1+\alpha)T\mathbf{x}$ is guaranteed to lie outside the interior
of $K$ for all $\mathbf{x}\in S(\Lambda_0)$.

The determinant of the dilated lattice $d(\Lambda')$ satisfies
$$
\begin{aligned}
&\frac{d(\Lambda')}{d(\Lambda_0)}=
(1+\alpha)^n \frac{d(\Lambda)}{d(\Lambda_0)}\\
&\le(1 + \sum_{\mathbf{x}\in S(\Lambda_0)}\upsilon_\mathbf{x}\rho_\mathbf{x}
 + C \max_{\mathbf{x}\in S(\Lambda_0)}|\rho_\mathbf{x}|^2)
(1+\varepsilon'\sum_{\mathbf{x}\in S(\Lambda_0)}|\rho_\mathbf{x}|)^n\\
&\le 1+ \sum_{\mathbf{x}\in S(\Lambda_0)}\upsilon_\mathbf{x}\rho_\mathbf{x} +
\varepsilon'\sum_{\mathbf{x}\in S(\Lambda_0)}|\rho_\mathbf{x}|\text,\end{aligned}
$$
where the quadratic and higher order terms have been absorbed into the last term.
\end{enumerate}
\end{proof}

\section{The case $n=3$}

It is known that the disk $B^2$ is not a local pessimum \cite{mahler95}. Therefore,
of the dimensions where the densest lattice packing of balls is known, Corollaries \ref{cor2} and
\ref{cor3} leave only the case $n=3$ to be resolved.
Since in all the other dimensions we have seen that $B^n$
is not a local pessimum, it may come as somewhat of a surprise, in spite
of Ulam's conjecture and the intuitive argument in the introduction, that $B^3$,
as we will prove in this section, turns out to be a local pessimum. 

The main idea of the proof comes from the fact that, as a consequence of Theorem \ref{cns1},
the critical determinant of a nearly spherical body is, in first order, determined by
the sum $\sum_{\mathbf{x}\in S(D_3)}[r(\mathbf{x})-1]$ of the change (compared to the sphere)
in the radial function evaluated at the twelve minimal vectors of $D_3$, oriented in such a way
as to minimize this sum. The average of this sum over all orientations of the minimal
vectors is proportional to the average change in the radial function over the sphere,
which is in turn proportional in first order to the change in volume of the body. If the sum
above is not a constant independent of the orientation of the minimal vectors, then
at some orientation it is smaller than the average and the optimal packing
fraction is increased in first order. It is not hard to show that
any even function for which the sum above is independent of the orientation must be
a second degree spherical harmonic up to a constant. In first order, such a change
to the radial function corresponds simply to a linear transformation.
Since we know that the packing fraction is invariant with respect to linear transformations,
we may pick a linear transformation that eliminates the second degree term of the harmonic expansion.

Any even continuous function $f$ on $S^2$ can be expanded in terms of spherical harmonics
$f(\mathbf{x}) = \sum_{l=0,l\text{ even}}^{\infty}f_l(\mathbf{x})$, where $f_l(\mathbf{x})$
is a homogeneous harmonic polynomial of degree $l$ in $\mathbf{x}\in\mathbb{R}^3$ restricted to $S^2$,
and the series converges at least in $L^2(\sigma)$ (let $\sigma$
denote the rotation-invariant probability measure on $S^2$). We may also write $f_l = \pi_l f$,
where $\pi_l$ is the orthogonal projection from $L^2(\sigma)$ to the finite dimensional
space of spherical harmonics of degree $l$.

\begin{lem}\label{quadlem}Given $\varepsilon>0$, there exists $\varepsilon'>0$ such that
if $K\subset \mathbb{R}^3$ satisfies $(1-\varepsilon')B^3\subseteq K\subseteq(1+\varepsilon')B^3$,
then $K$ has a linear image
$K'=TK$ that satisfies $(1-\varepsilon)B^3\subseteq K'\subseteq(1+\varepsilon)B^3$
and whose radial function has mean 1 and vanishing second spherical harmonic component.
\end{lem}

\begin{proof}
Let $\mathrm{Sym}_0^3$ be the space of zero-trace symmetric linear maps $\mathbb{R}^3\to\mathbb{R}^3$
equipped with Hilbert-Schmidt norm.
We identify $\mathrm{Sym}_0^3$ also with the space of zero-trace quadratic forms $\mathbb{R}^3\to\mathbb{R}$.
Consider the map $F_K:\mathrm{Sym}^3_0\to\mathrm{Sym}^3_0$,
given by $[F_K(A)](\mathbf{x}) = \langle\mathbf{x},A\mathbf{x}\rangle - \pi_2[r_{(1+A)K}](\mathbf{x})$,
where $A$ is viewed as a linear map and $F_K(A)$ is viewed as a quadratic form.
Note that the space of second degree spherical harmonics is the same as the space of 
zero-trace quadratic forms.

When $K$ is the unit ball $B^3$, we have
$$r_{(1+A)B^3}(\mathbf{x})=\frac{1}{||(1+A)^{-1}\mathbf{x}||}=1+\langle\mathbf{x},A\mathbf{x}\rangle +O(||A||^2)
\text{, }(\mathbf{x}\in S^2)$$
When taking the second harmonic component of the above, the constant $1$ vanishes,
the second term is preserved, and the third term contributes a term
that is again of order at most $||A||^2$. Therefore, $||F_{B^3}(A)||<C||A||^2$.
Fix a small closed ball $\mathcal{B}\subseteq\mathrm{Sym}_0^3$ around $0$
so that $F_{B^3}(\mathcal{B})\subseteq\frac{1}{2}\mathcal{B}$. 
Note that $F_K(A)$ can be made arbitrarily close to $F_{B^3}(A)$ uniformly with respect
to $A\in\mathcal{B}$ by appropriately choosing $\varepsilon'$.
Then if $\varepsilon'$ is small enough, $F_K$ maps $\mathcal{B}$ into $\mathcal{B}$,
and by Brouwer's fixed point theorem there is a fixed point $F_K (A)=A$.
Hence, the second spherical harmonic component of $r_{(1+A)K}$ vanishes. The mean of $r_{(1+A)K}$
might not be 1, but all that is left to do is to contract or dilate $(1+A)K$, and this operation
does not change the fact that the second spherical harmonic component vanishes.
By restricting $\varepsilon'$,
the norm of $A$ can be made as small as needed to guarantee that
$(1-\varepsilon)B^3\subseteq K'\subseteq(1+\varepsilon)B^3$.
\end{proof}

Let $\mathbf{p}$ represent a chosen
north pole on $S^2$. We say that a function
or measure on $S^2$ is \textit{zonal} if it is invariant under all rotations that
preserve $\mathbf{p}$. If $P_l(t)$ is the Legendre polynomial
of degree $l$, then $h_l(\mathbf{x}) = P_l(\langle\mathbf{x},\mathbf{p}\rangle)$
is the unique zonal harmonic of degree $l$ whose value at $\mathbf{p}$ is 1.
Given a function $f$ and a zonal measure $\mu$ on $S^2$,
we can define their convolution
$(\mu*f)(\mathbf{y}) = \int_{S^2} f(\mathbf{x})d\mu(U_\mathbf{y}(\mathbf{x}))$,
where $U_\mathbf{y}$ is any rotation that maps $\mathbf{y}$ to $\mathbf{p}$.
Such a convolution satisfies Young's inequality, and in particular as
a special case $||\mu*f||_1 \le \mu(S^2) ||f||_1$.
Convolution with a zonal measure acts as a multiplier transformation on the
harmonic expansion, i.e., $(\mu*f)(\mathbf{x}) = \sum_{l=0}^{\infty}c_l f_l(\mathbf{x})$.
The multiplier coefficients can be found by performing the convolution
on the zonal harmonic $c_l = (\mu*h_l)(\mathbf{p})=\int_{S^2} h_l(\mathbf{x})d\mu(\mathbf{x})$
\cite{convolutions}.

We now go on to prove two lemmas related to the configuration of minimal vectors
in $D_3$. Let $D_3$ be rotated so that $\mathbf{p}$ is one of the minimal vectors.
Then there is a unique zonal measure $\mu$ such that for every continuous zonal function $f$,
$$\int_{S^2}f(\mathbf{y})d\mu(\mathbf{y}) = \frac{1}{2} \sum_{\mathbf{x}\in S(D_3)} f(\mathbf{x})\text.$$
Since $\langle\mathbf{x},\mathbf{p}\rangle = 1,\tfrac{1}{2},0,-\tfrac{1}{2},-1$ with multiplicities
$1,4,2,4,1$ as $\mathbf{x}$ ranges over $S(D_3)$, we have that $\mu$ is
the zonal measure with total weight $6$, distributed
as follows: weight $1/2$ on each of the north and south poles,
weight $2$ on each line of latitude $30^\circ$ north and south,
and weight $1$ on the equator.

\begin{lem}\label{cllem}Let $c_l=P_l(1)+4P_l(\frac{1}{2})+P_l(0)=1+4P_l(\frac{1}{2})+P_l(0)$. Then
$c_l=0$ if and only if $l=2$. Moreover, $|c_l-1|<C l^{-1/2}$ for some constant
$C$.\end{lem}
\begin{proof}
It would here be advantageous to work with the rescaled Legendre polynomials
$Q_l(t) = 2^l l! P_l(t)$. Their recurrence relation is given by
$Q_{l+1}(t) = 2(2l+1) t Q_l(t) - 4l^2 Q_{l-1}(t)$. It is clear from the recurrence
relation and the base cases $Q_0(t)=1$ and $Q_1(t)=2t$
that the values of $Q_l(t)$ at $t=0,\frac{1}{2},1$ are integers. We are interested
in their residues modulo $8$. For $t=0$, $Q_{l+2}(0)/2^{l+2} = -(l+1)^2 Q_{l}(0)/2^l$, and so
by induction $Q_l(0)$ is divisible by $2^l$.
For $t=\frac{1}{2}$, $Q_{l+1}(\frac{1}{2})$ is odd
whenever $Q_{l}(\frac{1}{2})$ is odd, and hence by induction for all $l$. For $t=1$,
$Q_l(1)=2^l l! P_l(1)=2^l l!$ is also divisible by $2^l$. Therefore, for all $l\ge3$ we have
$2^l l! c_l=Q_l(0)+4Q_l(\frac{1}{2})+Q_l(1)\equiv 4 \pmod 8$
and $c_l$ cannot vanish. For $l=2$, we have $P_l(t)=\frac{3}{2}t^2-\frac{1}{2}$ and
$c_l = 1+4\cdot(-\frac{1}{8})-\frac{1}{2}=0$. The second part of the lemma
follows from the bound $|P_l(t)|<(\pi l\sqrt{1-t^2}/2)^{-1/2}$ \cite{bernstein}.
\end{proof}

\begin{lem}\label{philem}
Let $\mu$ be the zonal measure described above
and let $\Phi$ be the operator of convolution with $\mu$.
Then its multiplier coefficients are $c_l=1+4P_l(\frac{1}{2})+P_l(0)$. Further,
let $Z$ be the space, equipped with the $L^1(\sigma)$ norm, of even functions $f$ on $S^2$ for which $f_2=0$.
Then $\Phi$ maps $Z$ to $Z$, and as an operator $Z\to Z$ it is one-to-one, bounded, and has a bounded inverse.\end{lem}

\begin{proof}As noted above, the multiplier coefficients of a convolution are given by
$c_l = \int_{S^2} P_l(\langle \mathbf{x},\mathbf{p}\rangle) d\mu(\mathbf{x}) =
1+4P_l(\frac{1}{2})+P_l(0)$.
As a consequence of Lemma \ref{cllem}, $\Phi$ clearly maps $Z$ to $Z$ and
is one-to-one on $Z$. Since $\Phi$ is a convolution operator with
a finite measure it is a bounded operator $Z\to Z$.

We now construct an operator $\Psi:Z\to Z$, which we then show
to be the inverse of $\Phi$. Let 
$$\Psi = \sum_{k=0}^{3} (\idd-\Phi)^k
+\sum_{l\text{ even, }l\neq 2}c_l^{-1}(1-c_l)^4\pi_l\text.$$
The norm of $\pi_l$ induced by the $L^1$ norm does not exceed the $L^1$ norm of the projection kernel
$K_l(\mathbf{x},\cdot)$, but $||K_l(\mathbf{x},\cdot)||_1\le||K_l(\mathbf{x},\cdot)||_2=\sqrt{2l+1}$.
Therefore, the sum $\sum_{\text{l even, }l\neq 2}|c_l^{-1}|\cdot|1-c_l|^4\cdot||\pi_l||$ 
converges and so $\Psi$ is bounded. Note also that $\Psi$ is a multiplier
transform and its multiplier coefficients are simply $c_l^{-1}$ for all even $l\neq 2$.
Therefore $\Psi=\Phi^{-1}$.
\end{proof}

We can now prove the main result.

\begin{thm}\label{mainthm}
There exists $\varepsilon>0$ such that if $K\subset \mathbb{R}^3$ is a 
non-ellipsoidal origin-symmetric convex solid
and $(1-\varepsilon)B^3\subseteq K\subseteq(1+\varepsilon)B^3$, then
$\varphi(K)>\varphi(B^3)$. In other words, $B^3$ is a local pessimum for packing.
\end{thm}

\begin{proof}Given Lemma \ref{quadlem} and the
fact that $\varphi(K)$ is invariant under linear transformations of $K$,
we may assume without loss of generality
that $K$ is a non-spherical solid whose radial function has an
expansion in spherical harmonics of the form
$$r_K(\mathbf{x}) = 1 + \rho(\mathbf{x}) = 1+\sum_{l\text{ even},l\ge 4}\rho_l(\mathbf{x})\text.$$
Therefore, the volume of $K$ satisfies
$$\vll K =
\frac{4\pi}{3} \int_{S^2} r_K^3 d\sigma \ge
\frac{4\pi}{3} \left(\int_{S^2} {r_K} d\sigma\right)^3 =
\frac{4\pi}{3}\text.$$

We consider all the rotations $U(K)$ of the solid $K$ and
the determinants of the admissible lattices obtained when the construction of Theorem \ref{cns1} (ii)
is applied to $U(K)$. Note that the determinants obtained depend only
on $\rho(U^{-1}(\mathbf{x}))$ for $\mathbf{x}\in S(D_3)$.
Let us define $\Delta_K = \frac{\phi(K)^{-1}}{\phi(B^n)^{-1}}-1$.
Combining our bound on $\vll K$ with Theorem \ref{cns1} (ii) we get
\begin{align}
\Delta_K \le \min_{U\in SO(3)}
\left[\frac{1}{4} \sum_{\mathbf{x}\in S(D_3)} \rho(U^{-1}(\mathbf{x})) +
\varepsilon' \sum_{\mathbf{x}\in S(D_3)} |\rho(U^{-1}(\mathbf{x}))|\right]\text.\label{Delta1}
\end{align}
We may pick a single point $\mathbf{x}_0\in S(D_3)$ and decompose $SO(3)$ into subsets
$\mathcal{U}_\mathbf{y}$ of all rotations
such that $U^{-1}(\mathbf{x}_0)=\mathbf{y}$. In each subset $\mathcal{U}_\mathbf{y}$
the minimum on the right hand side of \eqref{Delta1} is no larger than the average
value over $\mathcal{U}_\mathbf{y}$
(with respect to the obvious uniform measure). This averaging procedure
transforms \eqref{Delta1} into
\begin{align}
\Delta_K \le \min_{\mathbf{y}\in S^2} \left[\frac{1}{2} \Phi[\rho](\mathbf{y}) + 2\varepsilon' \Phi[|\rho|](\mathbf{y})\right]\text,\label{Delta2}
\end{align}
where $\Phi$ is the same convolution operator as in Lemma \ref{philem}.

For integrable functions
$f$ and $g$ over a domain of measure 1, such that $f$ is of zero mean and $g$ is non-negative, we have
$\min[f+g]\le -||\min(f+g,0)||_1\le ||g||_1-||\min(f,0)||_1$.
Since $||f||_1=2||\min(f,0)||_1$, we have
$\min[f+g]\le-\frac{1}{2}||f||_1+||g||_1$. Therefore, if we let $f=\tfrac{1}{2} \Phi[\rho](\mathbf{y})$
and let $g=\varepsilon' \Phi[|\rho|](\mathbf{y})$, we obtain
\begin{equation*}
\Delta_K \le -\frac{1}{4} ||\Phi^{-1}||^{-1}\cdot||\rho||_1 + 2\varepsilon' ||\Phi||\cdot||\rho||_1\text.
\end{equation*}
Since $\varepsilon'$ can be made as small as needed by decreasing $\varepsilon$ we conclude that there
is a coefficient $c>0$ such that $\Delta_K \le -c ||\rho||_1$.
\end{proof}

\begin{rmk}Since the optimal lattice packing fraction of $B^3$ is also
its optimal packing fraction for non-lattice packings \cite{Hales}, and
since any ellipsoid can be packed at a higher packing fraction when not
restricted to a lattice \cite{ellipsoid}, it follows from
the theorem that $B^3$ is also a local pessimum among origin-symmetric
solids with respect to packing that allows rotations and non-lattice translations.
The case of origin-non-symmetric solids remains open.\end{rmk}

\begin{rmk}Note that the theorem provides a bound
for the improvement in the packing fraction of $K$ over
$B^3$ which is linear in terms of the $L^1$ norm of the difference
between the radial functions of $B^3$ and an appropriate
linear transformation of $K$.\end{rmk}

\begin{rmk}
It appears that
the size of the neighborhood $\varepsilon$ that can be obtained using 
this method of proof is
primarily governed by $||\Phi^{-1}||^{-1}$, which is in turn primarily
governed by $\min_{l\ge 4}|c_l|$. By numerical computation,
we find that $|c_l|>0.2$ for all even $l\ge4$, except for $l=10$,
for which $c_l\approx10^{-3}$.\end{rmk}

\textbf{Acknowledgments.} A special acknowledgment is owed to Fedor Nazarov
who contributed significantly to the work presented but declined co-authorship. In particular,
Nazarov contributed crucial ideas for the proofs of Lemma \ref{philem} and part (ii) of Theorem \ref{cns1},
provided a simpler proof of Lemma \ref{cllem} than in the original manuscript, and helped in the preparation of
a previous version of the manuscript.

\bibliographystyle{amsplain}
\bibliography{sphere}

\end{document}